\documentclass[a4paper,11pt]{article}
\usepackage[utf8]{inputenc}
\usepackage[T1]{fontenc}
\usepackage{graphicx}
\usepackage{amsmath,amsfonts,amssymb,amsthm}
\usepackage{mathtools}
\usepackage{mathrsfs}
\usepackage{bm}
\usepackage{color}
\usepackage{fullpage}

\definecolor{myred}{rgb}{0.77, 0.0, 0.1}
\definecolor{crimson}{rgb}{0.86, 0.08, 0.24}
\definecolor{awesome}{rgb}{1.0, 0.13, 0.32}
\definecolor{newgreen}{rgb}{0.0,0.6,0.0}
\definecolor{malachite}{rgb}{0.04, 0.85, 0.32}
\definecolor{pastelgreen}{rgb}{0.47, 0.87, 0.47}
\definecolor{myturq}{rgb}{0.1, 0.7, 0.7}

\usepackage{hyperref}
\hypersetup{
  colorlinks,
  linkcolor=blue, %
  citecolor=blue, %
  linktoc=all %
}

\usepackage{relsize}

\renewcommand{\leq}{\leqslant}
\renewcommand{\geq}{\geqslant}

\newcommand{\mleq}{\preccurlyeq}

\newcommand{\eps}{\varepsilon}

\newcommand{\argmin}{\mathop{\mathrm{arg}\,\mathrm{min}}}

\newcommand{\wt}{\widetilde}
\newcommand{\wh}{\widehat}

\newcommand{\R}{\mathbf R}

\usepackage{xspace}
 
\newcommand{\eg}{e.g.\@\xspace}
\newcommand{\iid}{i.i.d.\@\xspace}

\newcommand{\tr}{\mathrm{Tr}}

\newcommand{\opnorm}[1]{\| {#1} \|_{\mathrm{op}}}

\newcommand{\norm}[1]{\|#1\|}

\newcommand{\E}{\mathbb E}

\newcommand{\Var}{\mathrm{Var}}
\newcommand{\var}{\Var}

\newcommand{\cond}{|}%

\newcommand{\X}{\mathcal{X}}

\newcommand{\excessrisk}{\mathcal{E}}

\newtheorem{theorem}{Theorem}
\newtheorem{lemma}{Lemma}

\theoremstyle{definition}

\newtheorem{assumption}{Assumption}%

\theoremstyle{remark}

\title{An elementary analysis of ridge regression with random design}

\author{Jaouad Mourtada\footnote{CREST, ENSAE, Palaiseau, France;  \tt{jaouad.mourtada@ensae.fr}} \ \quad and \quad Lorenzo Rosasco\footnote{Universita’ di Genova \& Istituto Italiano di Tecnologia, Genova, Italy; Center for Brains, Minds and Machines,  MIT, Cambridge, United States; \tt{lorenzo.rosasco@unige.it}}}
\date{\today}

\begin{document}

\maketitle

\begin{abstract}
  In this note, we provide an elementary analysis of the prediction error of ridge regression with random design.
  The proof is short and self-contained.
  In particular, it bypasses the use of Rudelson's deviation inequality for covariance matrices, through a combination of exchangeability arguments, matrix perturbation and operator convexity.

\end{abstract}

\section{Introduction}
\label{sec:introduction}

Let $(X, Y)$ be a random vector in $\R^d \times \R$ with distribution $P$.
We consider the problem of random-design regression, namely prediction of $Y$ by linear functions of $X$.
(This setting also allows to consider nonlinear functions of general covariates $X'$, taking values in a measurable space $\X'$, by letting $X = \Phi (X')$ for some feature map $\Phi : \X' \to \R^d$.)
Specifically, the prediction error of a regression parameter $\theta \in \R^d$ is defined by its \emph{risk}
$L (\theta) = \E [ (Y - \langle \theta, X\rangle)^2 ]$, where $\langle \theta, x\rangle = \theta^\top x$ is the standard scalar product on $\R^d$.
In the statistical setting, the true distribution $P$, and in particular the (population) risk $L : \R^d \to \R^+$ and its minimizer $\theta^* = \argmin_{\theta \in \R^d} L (\theta)$, are unknown.
The aim is then, given a random independent and identically distributed (\iid) sample $(X_1, Y_1), \dots, (X_n, Y_n)$ from $P$, to find a good parameter $\wh \theta$, as measured by its \emph{excess risk}
\begin{equation}
  \label{eq:excess-risk}
  \excessrisk (\wh \theta)
  = L (\wh \theta) - L (\theta^*)
  \, .
\end{equation}

A popular approach to this problem is the method of regularized least squares (also called empirical risk minimization), where the estimator $\wh \theta$ minimizes the sum of an empirical error term and a regularization term favoring some structure on the parameter.
In this note, we consider the classical \emph{ridge} estimator~\cite{hoerl1962ridge,tikhonov1963regularization}, defined for $\lambda > 0$ by
\begin{equation}
  \label{eq:ridge-estimator}
  \wh \theta_\lambda
  := \argmin_{\theta \in \R^d} \bigg[ \frac 1 n  \sum_{i=1}^n (Y_i - \langle \theta, X _i \rangle)^2 +\lambda \|\theta\|^2 \bigg]
  = (\wh \Sigma_n + \lambda)^{-1} \cdot \frac{1}{n} \sum_{i=1}^n Y_i X_i
  \, ,
\end{equation}
where $\norm{\theta} := \langle \theta, \theta\rangle^{1/2}$ is the Euclidean norm and $\wh \Sigma_n := n^{-1} \sum_{i=1}^n X_i X_i^\top$ is the empirical covariance matrix.
While this estimator is well-studied (see Section~\ref{sec:discussion}), our aim here is to present a short and elementary analysis of its performance.
In particular, the analysis presented here does not rely on matrix concentration~\cite{rudelson1999random,ahlswede2002strong,oliveira2010sums,tropp2012user} or uniform deviation bounds for empirical processes~\cite{massart2000some,koltchinskii2006local,bartlett2005local}, but rather on a combination of exchangeability and matrix convexity arguments.
It draws inspiration from an analysis of~\cite{mourtada2022logistic} in the context of conditional density estimation.
Our main error estimate is provided in Theorem~\ref{thm:risk-bound-ridge} below.

\paragraph*{Notation.}

Given a $d \times d$ matrix $A$, we denote by $\tr (A)$ its trace and $\opnorm{A}$ its operator norm.
The $d\times d$ identity matrix is denoted $I_d$, or simply $I$; for $\lambda \in \R$, we denote $A + \lambda = A+ \lambda I$.
The symbol $\mleq$ denotes the standard order on symmetric matrices: $A \mleq B$ means that $\langle A v, v\rangle \leq \langle B v, v\rangle$ for all $v \in \R^d$.

\section{Risk analysis of ridge regression}
\label{sec:risk-analysis-ridge}

\paragraph*{Assumptions.}
The analysis requires two assumptions on the joint distribution $P$ of $(X, Y)$, the first one on the
distribution of the error $Y - \langle \theta^*, X\rangle$,
and the second one on the distribution of $X$.

\begin{assumption}
  \label{ass:wellspec}
  There exist $\theta^*\in \R^d$ and $\sigma>0$ such that
  \begin{equation}\label{eq:cond}
    \E [ Y \cond X ] = \langle \theta^*, X\rangle,\quad \text{and} \quad \var (Y \cond X) \leq \sigma^2.
  \end{equation}
\end{assumption}

The first condition in Assumption~\ref{ass:wellspec} states that the linear model is well-specified, in the sense that the true regression function $x \mapsto \E [Y|X=x]$ is linear.
This condition is standard, although restrictive when the dimension $d$ is low.
On the other hand, the guarantees we consider do not explicit depend on the dimension $d$, and extend with minor changes in notation to the case where $\R^d$ is replaced by an infinite-dimensional Hilbert space.
This allows to handle the case of reproducing kernel Hilbert spaces~\cite{aronszajn1950theory} (such as certain Sobolev spaces), for which ridge regression is a classical estimator~\cite{wahba1990spline,steinwart2008svm}.
When considering a ``universal'' kernel, the corresponding Hilbert space is dense in the space $L^2 (P_X)$ of square-integrable functions of $X$ \cite{steinwart2008svm}, in which case  the well-specified assumption is considerably less restrictive. The main issue is then the dependence of the bound on $\theta^*$.
In this respect, the bound of Theorem~\ref{thm:risk-bound-ridge} will only depend on $\theta^*$ through the approximation properties of balls of the Hilbert space.
Finally, the second condition in Assumption~\ref{ass:wellspec} is a bound on the conditional variance of $Y$ given $X$, which controls the amount of noise.
It holds for instance if $Y$ is bounded, or if the error $Y - \langle \theta^*, X\rangle$ is independent of $X$ with finite variance.

\begin{assumption}
  \label{ass:bounded-X}
   There exists a constant $R > 0$ such that $\| X \| \leq R$ almost surely.
\end{assumption}

The boundedness assumption~\ref{ass:bounded-X} is classical in the context of ridge regression.
This condition is automatically satisfied, for instance, in the case where $X$ is the feature associated to a bounded reproducing kernel Hilbert space~\cite{steinwart2008svm}.
Assumption~\ref{ass:bounded-X} implies in particular that the covariance matrix $\Sigma = \E [X X^\top]$ of $X$ is well-defined and satisfies $\tr (\Sigma) = \E \norm{X}^2 \leq R^2$.

\paragraph*{Risk analysis of ridge regression.}

The proof hinges on two main lemmas.
Before presenting them, we start with a classical bias-variance decomposition.

\begin{lemma}[Error decomposition]
  \label{lem:bias-var-ridge}
  Under Assumptions~\ref{ass:wellspec} and~\ref{ass:bounded-X}, we have for every $\lambda>0$,
  \begin{equation}
    \label{eq:bias-var-ridge}
    \E [ \excessrisk (\wh \theta_\lambda) ]
    \leq \lambda^2 \E [ \langle (\wh \Sigma_n + \lambda)^{-1} \Sigma (\wh \Sigma_n + \lambda)^{-1} \theta^*, \theta^*\rangle ]
    + \frac{\sigma^2}{n} \cdot \E [ \tr \{ (\wh \Sigma_n + \lambda)^{-1} \Sigma \} ]
    \, .
  \end{equation}
\end{lemma}

\begin{proof}
  Let  $Y = \langle \theta^*, X\rangle + \eps$, so that  $\E [ \eps \cond X ] = 0$ and $\E [ \eps^2 \cond X ] \leq \sigma^2$ almost surely. Then, 
  \begin{equation*}
    \wh \theta_\lambda
    = (\wh \Sigma_n + \lambda)^{-1} \frac{1}{n} \sum_{i=1}^n Y_i X_i
    = (\wh \Sigma_n + \lambda)^{-1} \wh \Sigma_n \theta^* + (\wh \Sigma_n + \lambda)^{-1} \frac{1}{n} \sum_{i=1}^n \eps_i X_i \, ,
  \end{equation*}
  so that
  \begin{equation*}
    \wh \theta_\lambda - \theta^*
    = - \lambda (\wh \Sigma_n + \lambda)^{-1} \theta^* + (\wh \Sigma_n + \lambda)^{-1} \frac{1}{n} \sum_{i=1}^n \eps_i X_i
    \, .
  \end{equation*}
  Moreover, it is standard (by Pythagoras' theorem in $L^2$) that $\excessrisk(\wh \theta_\lambda)= \| \wh \theta_\lambda - \theta^* \|_\Sigma^2$, where  $\|v\|^2_\Sigma=\langle \Sigma v, v\rangle$ for all $v\in \R^d$.
  Since $\E [ \eps_i \cond X_1, \dots, X_n ] = 0$ and $\E [ \eps_i^2 \cond X_1, \dots, X_n ] \leq \sigma^2$, then
  \begin{align*}
    \E [ \excessrisk(\wh \theta_\lambda) ]
    &= \E [ \| \wh \theta_\lambda - \theta^* \|_\Sigma^2 ] \\
    &= \lambda^2 \E \| (\wh \Sigma_n + \lambda)^{-1} \theta^* \|_\Sigma^2 + \frac{1}{n^2} \E \Big[ \sum_{i=1}^n  \| (\wh \Sigma_n + \lambda)^{-1} \eps_i X_i \|_\Sigma^2 \Big] \\
    &\leq \lambda^2 \E \| (\wh \Sigma_n + \lambda)^{-1} \theta^* \|_\Sigma^2 + \frac{\sigma^2}{n^2} \E \Big[ \sum_{i=1}^n  \tr \{ \Sigma (\wh \Sigma_n + \lambda)^{-1} X_i X_i^\top (\wh \Sigma_n + \lambda)^{-1} \} \Big] \\
    &= \lambda^2 \E [\langle (\wh \Sigma_n + \lambda)^{-1} \Sigma (\wh \Sigma_n + \lambda)^{-1} \theta^*, \theta^*\rangle] + \frac{\sigma^2}{n} \E [ \tr \{ (\wh \Sigma_n + \lambda)^{-1} \wh \Sigma_n (\wh \Sigma_n + \lambda)^{-1} \Sigma \} ]
      .
  \end{align*}
  The bound~\eqref{eq:bias-var-ridge} ensues by further bounding $(\wh \Sigma_n + \lambda)^{-1} \wh \Sigma_n (\wh \Sigma_n + \lambda)^{-1} \mleq (\wh \Sigma_n + \lambda)^{-1}$, and using that $\tr (A \Sigma) \leq \tr (B \Sigma)$ for symmetric matrices $A, B$ since $\Sigma$ is positive semi-definite.
\end{proof}
Lemma~\ref{lem:bias-var-ridge} shows the main quantities that need to be controlled in the random-design setting.
The first term in~\eqref{eq:bias-var-ridge} is a bias term, due to the use of a regularization favoring solutions with small norm. The second one is a variance term due to the presence of errors $\eps_i = Y_i - \langle \theta^*, X_i\rangle$.
Both terms depend on the (random) sample covariance matrix $\wh \Sigma_n$, as well as the population covariance matrix $\Sigma$.
The fact that both of these matrices appear comes from the fact that, in the random-design/statistical learning setting, one is interested in making a prediction at a new point $X$, rather than at the points $X_1, \dots, X_n$ in the sample.

In order to argue that the sample covariance matrix $\wh \Sigma_n$ is ``close'' to $\Sigma$ in a suitable sense and deduce an explicit error bound, some assumption on the distribution of $X$ is needed.
This is where Assumption~\ref{ass:bounded-X} is used.
Under this assumption, one can apply Rudelson's inequality for sample covariance matrices~\cite{rudelson1999random}; this is the approach adopted, for instance, in~\cite{caponnetto2007optimal,hsu2014ridge}.
Rudelson's inequality, a consequence of the non-commutative Khintchine inequality of~\cite{lustpiquard1991noncommutative}, is however a non-trivial result, despite subsequent simplifications to its proof~\cite{ahlswede2002strong,oliveira2010sums,tropp2012user}.
In addition, matrix concentration through Rudelson's inequality introduces an additional logarithmic term~\cite{tropp2012user,vershynin2012introduction}.

In what follows, we present an alternative approach to controlling the random matrix terms in the right-hand side of~\eqref{eq:bias-var-ridge}, which only uses short and elementary arguments.
The proof relies on a combination of exchangeability, matrix perturbation and operator convexity.
It draws inspiration from an analysis of~\cite{mourtada2022logistic}, where similar arguments were used in the context of conditional density estimation and logistic regression.

\begin{lemma}
  \label{lem:trace-df} 
  Under Assumption~\ref{ass:bounded-X}, we have that for every $\lambda > 0$, 
    \begin{equation}
      \label{eq:tr-random-design-upper-lower}
      \tr [ (\Sigma + \lambda I)^{-1} \Sigma ]
      \leq \E \tr [ (\wh \Sigma_n + \lambda I)^{-1} \Sigma ]
      \leq \Big( 1 + \frac{R^2}{\lambda n} \Big) \cdot \tr [ (\Sigma + \lambda I)^{-1} \Sigma ]
      \, .
    \end{equation}
  \end{lemma}

  \begin{proof}
    The lower bound comes from convexity of $A \mapsto \tr (A^{-1} \Sigma)$ over positive matrices (see Lemma~\ref{lem:operator-convexity-inverse} below) and Jensen's inequality.
    Let us now prove the upper bound.
    We start by writing:
    \begin{equation*}
      \E \big[ \tr ( (\wh\Sigma_n + \lambda I)^{-1} \Sigma ) \big]
      = n \, \E \big[ \langle (n \wh \Sigma_n + \lambda n I)^{-1} X_{n+1}, X_{n+1} \rangle  \big]
      \, ,
    \end{equation*}
    where $X_{n+1}$ is a random variable distributed as $X$ and independent of $X_1, \dots, X_n$.
    Now, the Sherman-Morrison identity~\eqref{eq:sherman-lemma} with $S = n \wh \Sigma_n + \lambda n I$ and $v = X_{n+1}$ shows that
    \begin{align*}
      &\langle (n \wh \Sigma_n + \lambda n I)^{-1} X_{n+1}, X_{n+1} \rangle \\
      &= \big( 1 + n^{-1} \langle (\wh \Sigma_n + \lambda I)^{-1} X_{n+1}, X_{n+1} \rangle \big) \langle (n \wh \Sigma_n + X_{n+1} X_{n+1}^\top + \lambda n I)^{-1} X_{n+1}, X_{n+1} \rangle \\
      &\leq \Big( 1 + \frac{R^2}{\lambda n} \Big) \big\langle \big( (n+1) \wh \Sigma_{n+1} + \lambda n I \big)^{-1} X_{n+1}, X_{n+1} \big\rangle
    \end{align*}
    where $\wh \Sigma_{n+1} = (n+1)^{-1} \sum_{i=1}^{n+1} X_i X_i^\top$, and where
    we used that, by Assumption~\ref{ass:bounded-X}, 
    $\langle (\wh \Sigma_n + \lambda I)^{-1} X_{n+1}, X_{n+1} \rangle \leq \| X_{n+1} \|^2 / \lambda \leq R^2/\lambda$.
    It follows that
    \begin{align}      
      \E \big[ &\tr ( (\wh\Sigma_n + \lambda I)^{-1} \Sigma ) \big] 
      \leq n \Big( 1 + \frac{R^2}{\lambda n} \Big) \E [ \langle ((n+1) \wh \Sigma_{n+1} + \lambda n I)^{-1} X_{n+1}, X_{n+1} \rangle ] \nonumber \\
      &= n \Big( 1 + \frac{R^2}{\lambda n} \Big) \cdot \frac{1}{n+1} \sum_{i=1}^{n+1} \E [ \tr \{ ((n+1) \wh \Sigma_{n+1} + \lambda n I)^{-1} X_i X_i^\top \} ] \label{eq:proof-trace-exchangeability} \\
      &= \Big( 1 + \frac{R^2}{\lambda n} \Big) \cdot \E [ \tr \{ ((1+1/n) \wh \Sigma_{n+1} + \lambda I)^{-1} \wh \Sigma_{n+1} \} ] \nonumber \\
      &\leq \Big( 1 + \frac{R^2}{\lambda n} \Big) \E [ \tr \{ (\wh \Sigma_{n+1} + \lambda I)^{-1} \wh \Sigma_{n+1} \} ] \nonumber \\
      &\leq \Big( 1 + \frac{R^2}{\lambda n} \Big) \tr [ (\Sigma + \lambda I)^{-1} \Sigma ] \label{eq:proof-trace-convexity}
    \end{align}
    where~\eqref{eq:proof-trace-exchangeability} follows from exchangeability of $(X_1, \dots, X_{n+1})$, while~\eqref{eq:proof-trace-convexity} follows from concavity of the map $A \mapsto \tr [ (A + \lambda I)^{-1} A ]$ over positive matrices (by Lemma~\ref{lem:operator-convexity-inverse}).
\end{proof}

Next, we turn to controlling the bias term in Lemma~\ref{lem:bound-matrix-bias} below.
The proof follows a similar recipe as that of Lemma~\ref{lem:trace-df}.

\begin{lemma}
  \label{lem:bound-matrix-bias}
  Under Assumption~\ref{ass:bounded-X}, we have that for every $\lambda>0$,
  \begin{equation}
    \label{eq:bound-operator-bias-ridge}
    \E [ (\wh \Sigma_n + \lambda)^{-1} \Sigma (\wh \Sigma_n + \lambda)^{-1} ]
    \mleq \Big( 1 + \frac{R^2}{\lambda n} \Big)^2 \lambda^{-1} (\Sigma + \lambda)^{-1} \Sigma
    \, .
  \end{equation}  
\end{lemma}

\begin{proof}
  Similarly to the proof of Lemma~\ref{lem:trace-df}, we start by writing:
  \begin{equation*}
    \E [ (\wh \Sigma_n + \lambda)^{-1} \Sigma (\wh \Sigma_n + \lambda)^{-1} ]
    = n^2 \, \E [ (n \wh \Sigma_n + \lambda n)^{-1} X_{n+1} X_{n+1}^\top (n \wh \Sigma_n + \lambda n)^{-1} ]
    \, .
  \end{equation*}
  Next, the Sherman-Morrison identity~\eqref{eq:sherman-lemma} applied to $S = n \wh \Sigma_n + \lambda n$ and $v = X_{n+1}$ implies that
  \begin{align*}
    &(n \wh \Sigma_n + \lambda n)^{-1} X_{n+1} X_{n+1}^\top (n \wh \Sigma_n + \lambda n)^{-1} \\
    &= (1 + \langle (n \wh \Sigma_n + \lambda n)^{-1} X_{n+1}, X_{n+1}\rangle)^2 (n \wh \Sigma_n + \lambda n + X_{n+1} X_{n+1}^\top)^{-1}  X_{n+1}  X_{n+1}^\top ( n \wh \Sigma_n + \lambda n + X_{n+1} X_{n+1}^\top)^{-1} \\
    &\mleq \Big( 1 + \frac{R^2}{\lambda n} \Big)^2 (n+1)^{-2} (\wh \Sigma_{n+1} + \lambda')^{-1} X_{n+1} X_{n+1}^\top (\wh \Sigma_{n+1} + \lambda')^{-1}
  \end{align*}
  with $\lambda' = \lambda n / (n+1)$, where we used that $\langle (\wh \Sigma_n + \lambda)^{-1} X_{n+1}, X_{n+1}\rangle \leq R^2 / \lambda$.
  It follows that, by exchangeability of $(X_1, \dots, X_{n+1})$,
  \begin{align*}
    \E [ (\wh \Sigma_n + \lambda)^{-1} \Sigma (\wh \Sigma_n + \lambda)^{-1} ]
    &= n^2 \E [ (n \wh \Sigma_n + \lambda n)^{-1} X_{n+1} X_{n+1}^\top (n \wh \Sigma_n + \lambda n)^{-1} ] \\
    &\mleq \Big( 1 + \frac{R^2}{\lambda n} \Big)^2 \frac{n^2}{(n+1)^2} \E \Big[ (\wh \Sigma_{n+1} + \lambda')^{-1} X_{n+1} X_{n+1}^\top (\wh \Sigma_{n+1} + \lambda')^{-1} \Big] \\
    &= \Big( 1 + \frac{R^2}{\lambda n} \Big)^2 \frac{n^2}{(n+1)^2} \frac{1}{n+1} \sum_{j=1}^{n+1} \E \Big[ (\wh \Sigma_{n+1} + \lambda')^{-1} X_{j} X_{j}^\top (\wh \Sigma_{n+1} + \lambda')^{-1} \Big] \\
    &= \Big( 1 + \frac{R^2}{\lambda n} \Big)^2 \frac{n^2}{(n+1)^2} \E \big[ (\wh \Sigma_{n+1} + \lambda')^{-1} \wh \Sigma_{n+1} (\wh \Sigma_{n+1} + \lambda')^{-1} \big] \\
    &\mleq \Big( 1 + \frac{R^2}{\lambda n} \Big)^2 \frac{n^2}{(n+1)^2} \lambda'^{-1} \E \big[ (\wh \Sigma_{n+1} + \lambda')^{-1} \wh \Sigma_{n+1} \big] \\
    &\mleq \Big( 1 + \frac{R^2}{\lambda n} \Big)^2 \frac{n^2}{(n+1)^2} \lambda'^{-1} (\Sigma + \lambda')^{-1} \Sigma 
  \end{align*}
  where the last inequality follows from operator concavity of $x \mapsto x (x + \lambda')^{-1}$ over $\R^+$ (a consequence of Lemma~\ref{lem:operator-convexity-inverse}).
  Inequality~\eqref{eq:bound-operator-bias-ridge} is then obtained after substituting $\lambda' = \lambda n / (n+1)$.
\end{proof}

Before deriving the excess risk bound, we express the bias term in a more relatable form.
\begin{lemma}
  \label{lem:reform-bias-ridge}
  For every $\lambda>0$,
  \begin{equation*}
    \lambda \| (\Sigma + \lambda)^{-1/2} \Sigma^{1/2} \theta^* \|^2
    = \inf_{\theta \in \R^d} \big\{ L (\theta) + \lambda \| \theta \|^2 \big\} - L (\theta^*)
    \, .
  \end{equation*}  
\end{lemma}
\begin{proof}
  Letting $\theta_\lambda = (\Sigma + \lambda)^{-1} \Sigma \theta^*$, direct computations show that:
  \begin{align*}
    \inf_{\theta \in \R^d} \big\{ L (\theta) + \lambda \| \theta \|^2 \big\} - L (\theta^*)
    &= \| \theta_\lambda - \theta^* \|_\Sigma^2 + \lambda \| \theta_\lambda \|^2 \\
    &= \lambda^2 \| \Sigma^{1/2} (\Sigma + \lambda)^{-1} \theta^* \|^2 + \lambda \| (\Sigma + \lambda)^{-1} \Sigma \theta^* \|^2 \\
    &= \big\langle \big( \lambda^2 (\Sigma + \lambda)^{-2} \Sigma + \lambda (\Sigma + \lambda)^{-2} \Sigma^2 \big) \theta^*, \theta^* \big\rangle \\
    &= \big\langle \lambda (\Sigma + \lambda)^{-1} \Sigma \theta^*, \theta^* \big\rangle \\
    &= \lambda \| (\Sigma + \lambda)^{-1/2} \Sigma^{1/2} \theta^* \|^2
      \, .
      \qedhere
  \end{align*}
\end{proof}

Finally, plugging Lemmas~\ref{lem:trace-df}, \ref{lem:bound-matrix-bias} and~\ref{lem:reform-bias-ridge} (the first for the variance term, the last two for the bias term) into the decomposition of Lemma~\ref{lem:bias-var-ridge},
we obtain the following excess risk bound. 

\begin{theorem}
  \label{thm:risk-bound-ridge}
  Under Assumptions~\ref{ass:wellspec} and~\ref{ass:bounded-X}, we have for every $\lambda>0$,
\begin{equation}
\label{eq:excessrisk-ridge}
  \E [ \excessrisk (\wh \theta_\lambda) ]
  \leq 
  \Big( 1 + \frac{R^2}{\lambda n} \Big)^2
  \inf_{\theta \in \R^d} \big\{ L (\theta) + \lambda \| \theta \|^2 - L (\theta^*) \big\}
  + \Big( 1 + \frac{R^2}{\lambda n} \Big) \frac{\sigma^2 \tr [ (\Sigma + \lambda )^{-1} \Sigma ]}{n}
  \, .
\end{equation}
\end{theorem}

\section{Discussion}
\label{sec:discussion}

\paragraph*{Comments on the bound~\eqref{eq:excessrisk-ridge}.}

For $\lambda \geq c R^2 / n$, the bound~\eqref{eq:excessrisk-ridge} is at most
\begin{equation*}
  C\cdot 
\Big(
 \inf_{\theta \in \R^d} \big\{ L (\theta) + \lambda \| \theta \|^2 - L (\theta^*) \big\}
 + \frac{\sigma^2 \tr [ (\Sigma + \lambda )^{-1} \Sigma ]}{n} \Big) \, ,
\end{equation*}
where $c, C$ are constants.
In addition, the first term above is at most $\lambda \| \theta^* \|^2$ (take $\theta = \theta^*$ instead of $\inf_\theta$).
A bound in finite dimension can be deduced by letting $\lambda \asymp R^2 / n$ and bounding $d_\lambda := \tr [ (\Sigma + \lambda)^{-1} \Sigma ] \leq d$, which yields a $O ( (\sigma^2 d + R^2 \| \theta^* \|^2) / n )$ bound.

Both the variance and the bias term in~\eqref{eq:excessrisk-ridge} are distribution-dependent; they depend, respectively, on the spectrum of $\Sigma$ (through the effective dimension $d_\lambda = \tr [ (\Sigma + \lambda)^{-1} \Sigma ]$) and on the approximation properties of Euclidean balls of $\R^d$.

As shown in the lower bound of Lemma~\ref{lem:trace-df}, the upper bound on the variance term from the decomposition of Lemma~\ref{lem:bias-var-ridge} is sharp up to universal constants in the regime $n \gtrsim R^2/\lambda$.
We note that the variance term from Lemma~\ref{lem:trace-df} is itself only an upper bound on the actual variance
(after bounding $(\wh \Sigma_n + \lambda)^{-1}\wh \Sigma_n \mleq I$),
though it is generally of the correct order (an exception is the     ``interpolation'' regime~\cite{bartlett2020benign}, where $\lambda$ is very small or equal to $0$ and $n \ll d$).
For instance, under the ``nonparametric'' regime $d \gg n$ and under a polynomial decay of eigenvalues of $\Sigma$, namely if $\tr (\Sigma^{1/b}) \leq B$ for some $b > 1$ and $B > 0$, then $d_\lambda \leq 2 B \lambda^{-1/b}$ and this gives the optimal variance under this assumption~\cite{caponnetto2007optimal}.

The bound on the bias term is somewhat less accurate (in particular, there is no matching lower bound in Lemma~\ref{lem:bound-matrix-bias}), though still of correct order in some relevant regimes.
Ideally, one may wish to replace $\wh \Sigma_n$ by $\Sigma$ (at least for $n$ large enough) in the bias term of Lemma~\ref{lem:bias-var-ridge}, leading to a term of
\begin{equation*}
  \lambda^2 \langle (\Sigma + \lambda)^{-1} \Sigma (\Sigma + \lambda)^{-1} \theta^*, \theta^*\rangle
  = L (\theta_\lambda) - L (\theta^*)
  \, ,
\end{equation*}
where $\theta_\lambda = \argmin_{\theta \in \R^d} \{ L (\theta) + \lambda \norm{\theta}^2 \} = (\Sigma + \lambda)^{-1} \Sigma \theta^*$.
Instead, the bound~\eqref{eq:excessrisk-ridge} gives a term $L (\theta_\lambda) - L (\theta^*) + \lambda \norm{\theta_\lambda}^2$, with an additional $\lambda \norm{\theta_\lambda}^2$ component.
Roughly speaking, this extra term dominates $L (\theta_\lambda) - L (\theta^*)$ when $\theta^*$ is highly aligned with the leading eigenvectors of $\Sigma$.
Similarly to the variance term, a simple way to assess this bound is to consider the stylized nonparametric regime, with $d$ large or infinite, and polynomial decay (this time, of coefficients of $\theta^*$ in the basis of eigenvectors of $\Sigma$).
Specifically, assume as in~\cite{caponnetto2007optimal} that $\norm{\Sigma^{(1-r)/2} \theta^*} \leq \rho$ for some $r > 0$ and $\rho > 0$; the parameter $r> 0$ controls the rate of decay of components of $\theta^*$.
One can bound the bias term as $C(\rho) \lambda^{\min(r,1)}$, while it is known (\eg, from~\cite{caponnetto2007optimal}) that the actual bias of ridge can be bounded as $\wt C(\rho) \lambda^{\min (r, 2)}$ under these assumptions\footnote{Different estimators can give better bias terms, see \eg \cite{blanchard2018optlip} and references therein.}.
The bound is therefore of the correct order for $0 < r \leq 1$, but suboptimal in the regime $r > 1$.

\paragraph*{Additional comments and references.}

A possible approach to analyzing ridge regression is based on viewing it as empirical risk minimization and using tools from empirical process theory together with localization and fixed-point arguments \cite{vandegeer1999empirical,massart2000some,bartlett2005local,koltchinskii2006local}, see for instance \cite[Example~2 p.~86]{koltchinskii2011oracle},  \cite{mendelson2003performance} and \cite{vaskevicius2020suboptimality} for analyses in this spirit.

A direct approach is based on matrix concentration~\cite{rudelson1999random,lustpiquard1991noncommutative,ahlswede2002strong,oliveira2010sums,tropp2012user}.
Results in this direction were first derived in \cite{devito2005model} and then refined in a series of works \cite{devito2005lip, smale2005shannon2, smale2007integral,caponnetto2007optimal}.
In particular, optimal bounds depending on the effective dimension %
$d_\lambda$
were first derived in \cite{caponnetto2007optimal}, see also \cite{hsu2014ridge,steinwart2009optimal,blanchard2018optlip}.
In fact, two-sided matrix concentration is not necessary to control the error, and one-sided lower bounds on the sample covariance matrix suffice, see~\cite{oliveira2016covariance,lecue2016performance,catoni2016pac,mourtada2019leastsquares,zhivotovskiy2021dimension} for results of this type.
In a different direction, a risk analysis of ridge regression (assuming that $X$ is a sub-Gaussian random vector, see~\cite{koltchinskii2017operators} for matrix concentration results in this case) covering more regimes of choices of $\theta^*,\Sigma, \lambda, n$ can be found in~\cite{tsigler2020benign}.

The elementary analysis of ridge regression presented here does not explicitly rely on matrix concentration (or lower tail) results.
The main arguments, namely exchangeability, matrix perturbation and matrix convexity, were also used in~\cite{mourtada2022logistic}, but for different estimators and for conditional density estimation rather than regression.
For ridge regression, an analysis related to the one presented in this note, based on average stability arguments, is proposed in~\cite{vaskevicius2020suboptimality}.
Additional relevant works include~\cite{koren2015expconcave,gonen2018stability}, that use average stability to analyze empirical minimization in exp-concave statistical learning (with bounds depending on the dimension $d$).
It is worth noting that, while exchangeability/leave-one-out/average stability arguments typically lead to simple and direct proofs, a shortcoming of these approaches is that they generally give in-expectation rather than deviation bounds.

Finally, a precise in-expectation finite-sample analysis of estimators based on stochastic gradient descent (SGD) can be found in~\cite{dieuleveut2016nonparametric} (extending results in~\cite{bach2013nonstrongly}), see also~\cite{yao2007early,rosasco2015incremental,jain2018par} (and references therein) for more information on SGD for least squares.
The arguments of~\cite{dieuleveut2016nonparametric} are direct and do not rely on matrix concentration either.
This analysis is however specific to iterative methods such as stochastic gradient descent, and it is unclear whether it applies to ridge regression.

\appendix
\section{Operator convexity and Sherman-Morrison's identity
}
\label{sec:appendix-lemmas}

In this appendix, we provide for the sake of completeness statements and short proofs of facts used in the proofs of Lemmas~\ref{lem:trace-df} and~\ref{lem:bound-matrix-bias}. 

We start with (operator) convexity of the matrix inverse.
\begin{lemma}[Lemma~2.7 in \cite{carlen2010trace}]
  \label{lem:operator-convexity-inverse}
  Let $S$ be a symmetric, positive semi-definite matrix.
  Then, the map
  \begin{equation*}
    A \mapsto \tr (A^{-1} S)
    \, ,
  \end{equation*}
  defined on the cone of positive-definite matrices, is convex.
\end{lemma}

\begin{proof}
  By continuity of the map $A \mapsto \tr (A^{-1} S)$ on its domain, it suffices to prove that it is midpoint-convex.
  Hence, it suffices to show that, for any positive-definite matrices $A, B$,
  \begin{equation*}
    \Big( \frac{A + B}{2} \Big)^{-1}
    \mleq \frac{A^{-1} + B^{-1}}{2}
    \, .
  \end{equation*}
  Now, letting $C = A^{-1/2} B A^{-1/2}$, since
  \begin{equation*}
    \Big( \frac{A + B}{2} \Big)^{-1}
    = A^{-1/2} \Big( \frac{I + C}{2} \Big)^{-1} A^{-1/2}
  \end{equation*}
  and
  \begin{equation*}
    \frac{A^{-1} + B^{-1}}{2}
    = A^{-1/2} \Big( \frac{I + C^{-1}}{2} \Big) A^{-1/2}
    \, ,
  \end{equation*}
  it suffices to show that $(I + C)^{-1}/2 \mleq (I + C^{-1})/2$.
  Up to conjugating with a rotation, one may assume that $C$ is diagonal.
  In this case, the desired inequality follows from the convexity of the (scalar) inverse on $\R_+^*$, applied to each entry of the diagonal.
\end{proof}

We also use the following identity involving the inverse of rank-one perturbations of matrices, which follows from the Sherman-Morrison identity.

\begin{lemma}
  \label{lem:sherman-morrison}
  Let $S$ be a positive-definite $d \times d$ matrix, and $v \in \R^d$.
  Then, one has
  \begin{align}
    \label{eq:sherman-lemma}
    S^{-1} v
    &= \big( 1 + \langle S^{-1} v, v\rangle \big) (S + v v^\top)^{-1} v
      \, .
  \end{align}
\end{lemma}

\begin{proof}
  We recall the Sherman-Morrison identity:
  \begin{equation*}
    (S + v v^\top)^{-1}
    = S^{-1} - \frac{S^{-1} v v^\top S^{-1}}{1 + \langle S^{-1} v, v\rangle}
    \, ,
  \end{equation*}
  which can be checked by multiplying both sides by $S+ v v^\top$,
  and deduce that
  \begin{equation*}
    (S + v v^\top)^{-1} v
    = S^{-1} v - \frac{\langle S^{-1} v, v\rangle}{1 + \langle S^{-1} v, v\rangle} S^{-1} v 
    = \frac{S^{-1} v}{1 + \langle S^{-1} v, v\rangle}
    \, .
    \qedhere
  \end{equation*}
\end{proof}

\paragraph{Acknowledgements.}
We would like to thank Francis Bach for encouraging us to share this
note and for many helpful comments.
L.R.~acknowledges the financial support of the European Research Council (grant SLING 819789), the AFOSR projects FA9550-18-1-7009, FA9550-17-1-0390 and BAA-AFRL-AFOSR-2016-0007 (European Office of Aerospace Research and Development), and the EU H2020-MSCA-RISE project NoMADS - DLV-777826.

{\small
  \bibliographystyle{abbrv}
  
}

\end{document}